\documentclass[12pt,a4paper]{amsart}
\usepackage{amsmath,amscd}
\usepackage{latexsym,amsmath,amssymb,mathrsfs,setspace,enumerate}
\usepackage[T1]{fontenc}
\usepackage[utf8]{inputenc}
\usepackage{lmodern}
\usepackage{amssymb,graphics}
\usepackage{pgf}
\usepackage{tikz}
\usetikzlibrary{arrows,automata}

\newtheorem{lemma}{Lemma}

\newtheorem{prop}{Proposition}

\newtheorem{thm}{Theorem}

\newtheorem{rem}{Remark}

\newtheorem{exam}{Example}

\newtheorem{dfn}{Definition}

\title{\textbf{ DIRECTED GRAPHS OF CAYLEY FUNCTIONS }}
\author {  Lejo J. Manavalan, P.G. Romeo }

\address {  Department of Mathematics,Cochin University of Science and Technology, Kochi, Kerala, INDIA. }
\email {  $lejojmanavalan171@gmail.com $ \\ }

\begin{document}

\begin{abstract}
In this paper we describe a condition under which a given function  that commute with an idempotent function on an infinite set is a Cayley function using its functional digraph.

  \end{abstract}

\maketitle
\section{Introduction}
Let $S$ be a non-empty set. A binary operation$(*)$ on a set $S$ is a mapping of $S\times S$ into $S$, where $S\times S$ is the set of all ordered pairs of elements of $S$. A binary operation on a set $S$ is said to be associative if $(a*b)*c = a*(b*c)$ for every $a$,$b$,$c$ in $S$. A semigroup is a system $(S,*)$ of a non-empty set $S$ together with an associative binary operation on $S$, with each element  $a$ in a  semigroup we can associate a transformation $\rho_a$ of $S$ defined by $x(\rho_a)=x*a  $ for all $x$ in $S$.  $\rho_a $ is called  the inner right translation of  $S$ by the element $a$  in  $S$. Similarly $\lambda_a$ where $x(\lambda_a)=a*x$ is called the inner left translation of $S$. If $(S, *)$ is a semigroup then $(S,.)$ where $. : S\times S \longrightarrow S$ by $a.b=b*a$ is also a semigroup and the inner left translation of $(S,*)$ will be  the right inner translation of $(S,.)$ and similarly the inner right translation of $(S,*)$ will be  the  inner left translation of $(S,.)$.

 A function $\alpha: S  \longrightarrow S$ is a Cayley function if there is a an associative binary operation on $S$ such that $\alpha(x)=a*x$ for all $x\in S$ for some $a \in S$. So a function is a Cayley function if it represents an inner translation of some semigroup. In 1971 Zupnik \cite{Zup} was able to identify whether a function was a Cayley function or not in terms of the powers of the given  function. And, lately in 2017 Ara$\acute{u}$jo et all in \cite{Ara} classified the Cayley function using functional digraphs.

Let $(S,*)$ be a semigroup then a left inner translation  must commute with every right inner translation of the semigroup and a right inner translation must commute with every left inner translation of the semigroup.  So finding Cayley functions that commute with another Cayley function is equivalent to finding the possible left inner translations  of a semigroup when one has a right inner translation or vice-versa.  Ara$\acute{u}$jo et all in \cite{Ara} outlined a six step process to study semigroups in which the third step is to identify the Cayley functions that commute with another Cayley function.  Further in his paper Ara$\acute{u}$jo et all  also classifies the Cayley function that commute with a finite permutation leaving open the general case.

In the following we describes the Cayley functions that commute with an infinite idempotent.  In other words, we describe the candidates for the rows of the Cayley table of a finite semigroup when one of its columns is an idempotent. 

\section{Preliminaries}
In the following  we recall the definitions and results of some preliminaries regarding the functional digraphs of Cayley functions needed in the sequel. For a non-empty set $S$ let $T(S)$ be the set of all the functions on $S$ (full transformations on $S$). A directed graph (or a digraph) is a pair $D = (S, \rho)$ where $S$ is a non-empty set of vertices (not necessarily finite), which we denote by $V (D)$, and any pair $(x, y) \in \rho$ is called an arc of $D$, which we write as $x \rightarrow y$. A vertex $x$ is called an initial vertex in $D$ if there is no $y \in S$ such that $y \rightarrow x$; A vertex $x$ is called a terminal vertex in $D$ if there is no $y \in S$ such that $x \rightarrow y$.

A digraph $D$ is called a functional digraph if there is $\alpha \in T(S)$ such that for all $x, y \in S$, $x \rightarrow y$ is an arc in $D$ if and only if $\alpha(x) = y$. If such an $\alpha$ exists, then it is unique, and we write $D = D_\alpha$ which is  the digraph that represents $\alpha$. 
Let D be a digraph and $...,x_0,x_1,...$ be pairwise distinct vertices of $D$. Then the following sub-digraphs :
\begin{enumerate}
\item $x_0 \rightarrow x_1 \rightarrow .......\rightarrow x_{k-1} \rightarrow x_0$ is called a cycle of length $k$ $(k \geq 1)$, denoted by $(x_0 x_1 . . . x_{k-1})$

\item
$x_0 \rightarrow x_1 \rightarrow ....... \rightarrow x_m$ is called a chain of length $m$, denoted by $[x_0 x_1 ... x_m]$ $(m \geq 0)$

\item
$x_0 \rightarrow x_1 \rightarrow x_2 \rightarrow .......$ is called a right ray, denoted by $ [x_0 x_1 x_2 ...\rangle;$ 

\item
$......\rightarrow x_2 \rightarrow x_1 \rightarrow x_0$ is called a left ray, denoted by  $\langle ... x_2 x_1 x_0]$

\item
$....\rightarrow x_{-1} \rightarrow x_0 \rightarrow x_{1} \rightarrow .....$ is called a double ray, denoted by  $\langle... x_1 x_0 x_1 ...\rangle$
\end{enumerate}

 Let $D_\alpha$ be a functional digraph, where $\alpha \in T(S)$. A right ray $[x_0 x_1 x_2 . . .\rangle$ in $D_\alpha$ is called a maximal right ray if $x_0$ is an initial vertex of $D_\alpha$.  A leftray $L=\langle...y_2 y_1 y_0]$ in $D_\alpha$ is called an infinite branch of a cycle $C$ in $D_\alpha$ if $x_0$ lies on $C$ and $x_1$ does not lie on $C$. Similarly a left ray $L=\langle...y_2 y_1 y_0]$ is a infinite branch of  double ray $W$ if  $y_0$ lies on $W$ and $y_1$ does not lie on $W$.  We will refer to any such $L$ as an infinite branch in $D_\alpha$.
 
   A chain $P = [x_0 x_1 ... x_m]$ of length $m \geq 1$ in $D_\alpha$ is called a finite branch of a cycle $C$ in $D_\alpha$ if $x_0$ is an initial vertex of $D_\alpha$, $x_m$ lies on $C$ and $x_{m-1}$ does not lie on $C$. Similarly one can  define a finite branch of  double ray $W$, maximal right ray $R$ and infinite branch $L$. A chain $P = [x_0 x_1 ... x_m]$ of length $m \geq 1$ in $D_\alpha$ is called a finite branch of an infinite branch $L = \langle... y_2 y_1 y_0]$,  if $x_0$ is an initial vertex of $D_\alpha$, $x_m$ lies on $L$ and $x_{m-1}$ does not lie on $L$ such that $x_m \neq y_0$. We will refer to any such $P$ as a finite branch in $D_\alpha$. double ray $W$, maximal right ray $R$, infinite branch $L$.
By a branch in $D_\alpha$ we will mean a finite or infinite branch in $D_\alpha$. Note that all branches of a maximal right ray $R$ or an infinite branch $L$ are finite. In other words, we only consider infinite branches of cycles and double rays.

  Let $ \alpha \in T(S)$  $x \in S$. The subgraph of $D_\alpha$ induced by the set $$\{y \in  S : \alpha^k(y) = \alpha^m(x)  \text{ for some integers  } k, m \geq 0 \} $$is called the component of $D_\alpha$ containing $x$.  


The following  proposition,  describe  functional digraphs.

\begin{prop} \cite{Har} Let $D_\alpha$ be a functional digraph. Then for every component $A$ of $D_\alpha$ exactly one of the following three conditions holds: \begin{enumerate}
\item A has a unique cycle but not a double ray or right ray; (where the component is the join of the cycle and its branches) 
\item A has a double ray but not a cycle (where the component is the join of the double ray and its branches ) ; or
\item A has a maximal right ray but not a cycle or double ray(where the component is the join of the right ray and its finite branches).
\end{enumerate}
	 \end{prop}
Suppose that a component $A$ of $D_\alpha$ has a right ray $R$ but not a double ray, then $A$ is the join of its maximal right rays we  call such a component $A$ is of type rro (right rays only).	  If we consider the functional digraph  on a finite set then the directed graph will have only components that are union of a cycle and its finite branches.

 \begin{dfn} Let $\alpha \in T(S)$. The stable image of $\alpha$ denote $sim(\alpha)$ is a subset of $S$ defined by 
$$sim(\alpha) = \{x \in S : x \in img(\alpha^n) \text{  for every  } n \geq 0\}$$ 
\end{dfn}
For $\alpha \in T(S)$ we have the following: \begin{itemize}
	\item $sim(\alpha)$ consists of the vertices of $D_\alpha$ that lie on cycles, double rays, or infinite branches; 
	\item $ sim(\alpha) = \phi$ if and only if each component of $D_\alpha$ is of type $rro$. \end{itemize}
\begin{dfn}
 The stabilizer of $\alpha \in T (S)$ is the smallest integer $s\geq 0$ such that $img(\alpha^s) = img(\alpha^{s+1})$. If no such $s$ exists then $\alpha$ has no stabilizer. 
 \end{dfn}
 
 The following are certain properties of functional digraphs representing transformations that have the stabilizer. 

  If $\alpha \in T(S)$  then: \begin{itemize}
	\item the stabilizer of $\alpha$ is the smallest integer $s \geq 0$ such that $\alpha^s(x) \in  sim(\alpha)$ for every $x \in S$; 
	\item has the stabilizer $s = 0$ if and only if $img(\alpha) = sim(\alpha) = S$, which happens if and only if each component $A$ of $D_\alpha$ is either the join of a cycle $C$ and the infinite branches of $C$ or the join of a double ray $W$ and the infinite branches of $W$ ;
	\item if $\alpha$ has the stabilizer $s$, then $sim(\alpha) = img(\alpha^s)$.
	\end{itemize}
	 Note that a  transformation may have a non-empty stable image and no stabilizer. 
	 
	 \begin{exam} Consider the function $\alpha$ on the set $S = \{...,x_{-1}, x_0, x_1,.....,y_0, y_1, y_2......\}$ whose directed graph is the following,
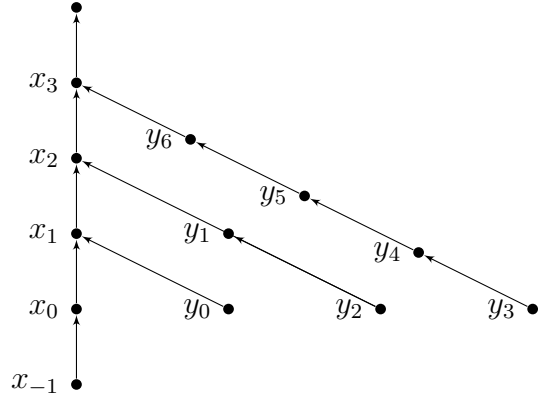
\begin {center}

\begin{tikzpicture}

\tikzset{vertex/.style = {shape=circle,fill=black,minimum size=4pt,
                            inner sep=0pt}}
\tikzset{edge/.style = {->,> = latex'}}
 \node[vertex] (n0) at (0,-1)  [label=left:$x_{-1}$] {};
 \node[vertex] (n1) at (0,0)  [label=left:$x_0$] {};
 \node[vertex] (n2) at (0,1) [label=left:$x_1$] {};
  \node[vertex] (n3) at (0,2)  [label=left:$x_2$] {};
   \node[vertex] (n4) at (0,3)  [label=left:$x_3$] {};
    \node[vertex] (n5) at (0,4)  [label=left:$$] {};
  
     \node[vertex] (n11) at (2,0)  [label=left:$y_0$] {};
      \node[vertex] (n12) at (2,1)  [label=left:$y_1$] {};
       \node[vertex] (n13) at (4,0)  [label=left:$y_2$] {};
        \node[vertex] (n14) at (6,0)  [label=left:$y_3$] {};
         \node[vertex] (n15) at (4.5,.75)  [label=left:$y_4$] {};
          \node[vertex] (n16) at (3,1.5)  [label=left:$y_5$] {};
           \node[vertex] (n17) at (1.5,2.25)  [label=left:$y_6$] {};
 
\draw[edge] (n0) to (n1);
\draw[edge] (n1) to (n2);
\draw[edge] (n2) to (n3);
\draw[edge] (n3) to (n4);
\draw[edge] (n4) to (n5);
\draw[edge] (n11) to (n2);
\draw[edge] (n13) to (n12);
\draw[edge] (n13) to (n3);
\draw[edge] (n14) to (n15);
\draw[edge] (n15) to (n16);
\draw[edge] (n16) to (n17);
\draw[edge] (n17) to (n4);

\end{tikzpicture}
\end{center} Then $\{...,x_{-1}, x_0, x_1,.....,\}$ is the stable image of $\alpha$. Since the length of the branches keeps on increasing one cannot find an $s$ such that $img(\alpha^s) = img(\alpha^{s+1})$  \end{exam}
 
  \begin{dfn} For $\alpha \in T(S)$ a finite branch $[x_0 x_1 ... x_m]$ in $D_\alpha$ is called a twig in $D_\alpha$ if $x_m \in sim(\alpha)$ (that is, $x_m$ lies on a cycle, double ray, or infinite branch) and $x_p  \notin sim ( \alpha)$ for every  $p \in \{ 0 , . . . , m - 1 \}$ . 
\end{dfn}
Every twig is a branch but every finite branch need not be a twig. For example a finite branch of an infinite branch of a double ray will form a branch of the double ray but not a twig.

\begin{dfn}
 Let $\alpha \in T(S)$ has the stabilizer $s$ then define
\begin{equation}
\Omega_{\alpha} =  \begin{cases}   
 a \in S : \alpha^s(a) \in  sim(\alpha)    \text{ but }     \alpha^{ s-1}(a)  \notin sim(\alpha)  &  \mbox{if } s > 0 \\S & \mbox{if } s=0 \}.\end{cases} 
\end{equation}
\end{dfn}

Note that for $s > 0$, $\Omega_\alpha$ consists of the initial vertices of the twigs of length $s$ in $D_\alpha$. 

\begin{thm} \cite{Zup} Let $ \alpha \in T(S)$. Then $\alpha$ is a Cayley function if and only if exactly one of the following conditions holds:\begin{enumerate}
\item has no stabilizer and there exists $a \in S$ such that $\alpha^n(a) \notin img(\alpha^{n+1})$ for every $n \geq 0$;
\item has the stabilizer $s$ such that $\alpha | img(\alpha^s)$ is one-to-one and there exists $a \in \Omega_\alpha  $ such that $\alpha^m(a) = \alpha^n(a) $ implies $\alpha^m = \alpha^n$ for all $m,n\geq 0$; or
\item has the stabilizer $s$ such that $\alpha | img(\alpha^s)$ is not one-to-one and there exists $a \in \Omega_\alpha $ such that:\begin{enumerate}
\item $\alpha^m(a)=\alpha^n(a)$ implies $m=n$ for all $m,n\geq 0$ ; and 
\item For every $n>s$, there are pairwise distinct elements $y_1, y_2, . . . $ of $T$ such that $\alpha(y_1) = \alpha^n(a)$, $\alpha(y_k) = y_{k-1}$  for every $k \geq 2$, and if $n > 0$ then $y_1\neq \alpha^ {n-1}(a) $
\end{enumerate}
\end{enumerate}
	 \end{thm}
	
\begin{rem}All idempotent functions are Cayley functions.
\end{rem}

The following theorem characterises Cayley functions  using their functional digraphs .

\begin{thm}

 \begin{enumerate}\cite{Ara2}
\item  Let $\alpha \in T(S)$ be such that $D_\alpha$ has a component of type rro. Then $\alpha$ is a Cayley function if and only if $D_\alpha$ has a component of type rro such that :\begin{enumerate}
\item it is the join of maximal right ray $[x_0 x_1 x_2.... \rangle$ and its branches;
\item for every $i\geq1$, if $[y_0 y_1 ....y_m =x_i]$ is a branch of R, then $m\leq i$.
\end{enumerate}

\item Let $\alpha \in T(S)$ be such that every component of $D_\alpha$ has a unique cycle or a double ray and $D_\alpha$ does not have an infinite branch. Then $\alpha$ is a Cayley function if and only if the following conditions are satisfied:\begin{enumerate}
\item $s = sup_b(\alpha)$ is finite;
\item if $s>0$ and $D_\alpha$ has a double ray, then some double ray in $D_\alpha$ has a branch of length $s$
\item if $D_\alpha$ does not have a double ray, then there are integers $1\leq k_1 \leq k_2 \leq...\leq k_p$, $p\geq1$, such that \begin{enumerate}
\item $\{k_1,...,k_p\}$ is the set of the lengths of the cycles in $D_\alpha$;
\item $k_i$ divides $ k_p$ for every $ i\in \{1,....p\}$ and
\item if $s>0$, then some cycle of $D_\alpha$ of length $k_p$ has a branch of length $s$.
\end{enumerate}

\end{enumerate}

\item Let $\alpha \in T(S)$ be such that every component of $D_\alpha$ has a unique cycle or a double ray and $D_\alpha$ has an infinite branch. Then $\alpha$ is a Cayley function if and only if the following conditions are satisfied:\begin{enumerate}
\item $s = sup_t (\alpha)$ is finite;
\item $D_\alpha$ has a double ray $W = \langle... x_{-1} x_0 x_1 ...\rangle$ such that for some $x_i$:\begin{enumerate}
\item if $s >0$ then $W$ has a finite branch at $x_i$ of length $s$;
\item $W$ has an infinite branch at each $x_j$ with $ j > i$.
\end{enumerate}

\end{enumerate}
\end{enumerate}
\end{thm}

\section{Directed graph $D_{\alpha}$  with $\alpha \in C(\epsilon)$}

Let $S$ be a non-empty set. For a transformation $\alpha \in $ $T(S)$, the centralizer $C(\alpha)$ of  $\alpha$ on a set $S$ is the set of all elements in $T(S)$ that commute with $\alpha$. In this section we discuss the properties of the directed graph of $\alpha$ so that $\alpha$ is Cayley function   and $\alpha \in C(\epsilon)$ for an  idempotent $\epsilon$. The centralizers in the full transformation semigroup have been studied in \cite{Ara2}.

 \begin{thm}\label{thme}\cite{Ara2} Let $\epsilon$,  $\alpha \in$ $T(S)$ and $\epsilon$ an idempotent. Then $\alpha \in C(\epsilon)$ if and only if for every connected component $\gamma $ of $\epsilon$ with cycle $(y)$ there exists  a connected component $\delta$ of $\epsilon$ with cycle $(z)$ such that $y \alpha= z$ and $(dom(\gamma))\alpha \subseteq dom(\delta)$.

\end{thm}

Let $\epsilon \in T(S)$ and let $C_\epsilon $ be the set of connected components of $\epsilon$.
For $\alpha \in C(\epsilon)$,  define a function $\Phi_\alpha^{\epsilon}$ on $C_\epsilon$ where 
 $$ \Phi_\alpha^{\epsilon}(\gamma) \text{ is the unique }\delta \in C_\epsilon \text{ such that} (dom(\gamma))\alpha \subseteq dom(\delta)  $$
 
Clearly $\Phi_\alpha^{\epsilon}$ is well defined since       im($ \alpha/ dom(\gamma )$) $\subseteq$ dom($\delta$) .

Let $ \epsilon \in T(S)$ be an idempotent and $C_\epsilon$ be the set of  connected components of $\epsilon$.  Each component of $ \epsilon$ will have a one cycle  and may or may not have any branches.

Observing  $D_{\Phi_\alpha^{\epsilon}}$, we can determine which cycle is mapped by $\alpha$ to which cycle but since a branch can be mapped by $\alpha$ to a cycle or a branch it is not precise from the graph $D_{\Phi_\alpha^{\epsilon}}$ .When the set $S$ is finite, if a branch of $\epsilon$ is mapped to a cycle of $\epsilon$ it induces a branch of length one, if not it will mapped to a branch which when mapped to a cycle form a branch of length 2 , otherwise it is again mapped to a branch. This process terminates when the branch is finally mapped by $\alpha$ to a cycle or when the branches form a cycle on its own. If $S$ is infinite and  $D_\epsilon$ has only finite number of connected components  then there will be at least one component that has infinite number of branches of length $1$. Also if $S$ is infinite and  $D_\epsilon$ has infinite number of connected components then each component may or may not have infinite number of branches.

\begin{lemma} Let $\alpha$, $\epsilon \in T(S)$, $\epsilon$ an idempotent and $\alpha \in C(\epsilon)$. Let $A$ be a connected component of $D_{\Phi_\alpha^{\epsilon}}$ with cycle $(C_0 ...C_{k-1})$,  $L$ be the set of all elements $x\in S$ such that $x$ is in some $\gamma \in A$ where $\gamma$ is a connected component  of $\epsilon$ and $Z$ is the set of all elements $x\in S$ such that $x$ is in some cycle of $\epsilon$.Then
\begin{enumerate}
\item if $\gamma$ , $\delta  \in A$ is such that $ \Phi_\alpha^{\epsilon}(\gamma)=\delta$, then for $x\in \gamma$, $\alpha(x)$ is in $\delta$
\item for every $x \in L$, $\alpha(x)$ is in  $L$
\item $\alpha/_Z$ $\in$ $T(Z)$ 
\item $\alpha/_L$ $\in$ $T(L)$
\item each $c_i$ in the cycle  has  length 1.
\end{enumerate}
\end{lemma}

\begin{lemma} Let $\alpha$ , $\epsilon \in T(S)$, $\epsilon$ an idempotent and $\alpha \in C(\epsilon)$ and that $A$ be a connected component of $D_{\Phi_\alpha^{\epsilon}}$ with cycle $(C_0 ...C_{k-1})$, $Z$ be the set of all elements $x\in S$ such that $x$ is the vertex of one cycles in some $C_i$ of $A$ and  that  $L$ be the set of all elements $x\in S$ such that $x$ is in some $\gamma \in A$. Then
\begin{enumerate}
\item $\alpha^k(x)=x$ for any $x $ in $Z$
\item then the cycle induced  in $D_{\alpha/ Z}$ has length $k$ 
\item the length of  cycles in $D_{\alpha/ L}$ ($D_{\alpha|_{L- Z}}$) is a multiple of $k$).

\end{enumerate}
\end{lemma}

\begin{proof} 
It is enough to prove 3, since $(C_0 ...C_{k-1})$ is a cycle of length $k$, $(\Phi_\alpha^{\epsilon})^k(C_0)=C_0$, so if a vertex $c_0^1$  of a branch say of $C_0$ is part of  a cycle in $\alpha$, then its minimum length is $k$ as it has to pass through k vertices before $c_0^1$ can posssibly be mapped to $c_0^1$. Now if $\alpha^k(c_0^1) \neq c_0^1$, $c_0^1$ has to pass through another k edges before it can be mapped by $\alpha$ to $c_0^1$, so every time a length  k is added. Hence cycles induced by the branches of $D_{\alpha/ L}$ is of length $m k$.
\end{proof}

If  $\epsilon \in T(S)$ is an idempotent   and $\alpha \in C(\epsilon)$ and  $A$ a connected component of $D_{\Phi_\alpha^{\epsilon}}$ with cycle $(C_0 ...C_{k-1})$, $Z$  the set of all elements $x\in S$ such that $x$ is in some cycle of $\epsilon$ and  that  $L$  the set of all elements $x\in S$ such that $x$ is in some $\gamma \in A$.  The cycle induced by cycles in $D_{\alpha/ L}$ has length $k$ and the branches cannot induce a cycle if any one of the component in the cycle has no branches, further there is a upper limit on length of the cycle $mk$ ($m$ = min\{$n_1,.....,n_{k-1}$\}  where $n_i$, is the number of branches in $c_i$ for $1\leq i \leq k-1$)

\begin{lemma}Let $\alpha$ and $\epsilon \in T(S)$ where $\epsilon$ is an idempotent and $\alpha \in C(\epsilon)$ and $A$ be a connected component of $D_{\Phi_\alpha^{\epsilon}}$ with cycle $(C_0 ...C_{k-1})$ such that at least one of the cycle $C_i$ has no branches. $L$ be the set of all elements $x\in S$ such that $x$ is in some $\gamma \in A$. $Z$ be the set of all elements $x\in S$ such that $x$ is the vertex of one cycles in some $C_i$ of $A$ and $s$ be the maximum of the length of the branches in $A$ ( s = 0 if $A$ has no branches ).  Then
\begin{enumerate}
\item The branches of $D_\epsilon$ in cycle $(C_0 ...C_{k-1})$ does not induce a cycle in $D_\alpha$
\item if $s=0$ then  the length $l$ of branches in $\alpha$ is such that $0\leq l\leq k-1$, where $k$  is the length of the cycle  $(C_0 ...C_{k-1})$
\item if $s\geq1 $ then the length $l$ of branches in $\alpha$ is such that $0\leq l\leq k+s$, where $k$  is the length of the cycle  $(C_0 ...C_{k-1})$
 \end{enumerate}
\end{lemma}
\begin{proof}
 Without loss of generality, assume that $ C_{k-1}$ have no branches. Let $c^1_0$ be a vertex of a  branch of $c_0$(a cycle in the connected component $C_0$).  If $c_0^1$ is mapped by $\alpha$ to the cycle in $C_1$ then it forms a branch of length 1 in  $\alpha$, (if not $c_0^1$ is  mapped a vertex say $c^1_1$  in   $c_{1})$. Further if $c_1^1$ is mapped to the cycle in $C_2$ it forms a branch of length 2 (otherwise $c_0^1$ is  mapped a vertex say $c^1_2$  in   $c_{2})$. The process  terminates when it reaches $ c_{k-1}$ as $ C_{k-1}$ has   no branches and hence the maximum length of the branch is $k-1$.
 
 To prove (2), let $ C_{k-1}$ have no branches and let $A$ have a branch of length $s$ adjoined to $c_0$ say  $R_0,R_1, R_2....,R_s= C_0$ is a branch of $A$. Let $r_i$ be the cycle in $R_i$ and $r_i^k$ a branch in $c_i$. If  a branch $r_0^1$ is mapped to cycle $r_1$ then it form another branch of length 1 to the branch induced by the cycles $r_0, ....r_s$ if not it is mapped to a branch say $c_1^k$, which if mapped to cycle $c_2$ forms a branch of length 2 to the branch induced by the cycles $c_0, ....c_s$ in $(C_0 ...C_{k-1})$ if not it is mapped to a branch say $c_2^k$. Proceeding like this it is seen that $0\leq l\leq k+s$.
\end{proof}

The following lemma is an improvement to Lemma 5.5 of \cite{Ara} and the proof is also  similar.
\begin{lemma}Let $\alpha$ and $\epsilon \in T(S)$ and $\alpha \in C(\epsilon)$. Let $A$ be a connected component of $D_{\Phi_\alpha^{\epsilon}}$ with cycle $(C_0 ...C_{k-1})$ such that at least one of the cycle $c_i$ has no branches.  $L$ be the set of all elements $x\in S$ such that $x$ is in some $\gamma \in A$.  $Z$ be  the set of all elements $x\in S$ such that $x$ is the vertex of one cycles in some $C_i$ of $A$. Let $s$ be the maximum length of the branches in $A$ ( s = 0 if $A$ has no branches ).  Then
\begin{enumerate}
\item if $s=0$ then  the length $l$ of branches is such that $0\leq l\leq m.k$, where $m = min\{m_i\}$ ,where $m_i$ is the maximum number of branches 
\item if $s\geq1 $ then the length $l$ of branches in $D_{\alpha|_{L}}$  is such that $0\leq l\leq (m.k)+s $, where $m = min\{m_i\}$ ,where $m_i$ is the number of branches in $C_i$\end{enumerate}
\end{lemma}

\begin{dfn} Let $\alpha \in T(S)$, $sup_{b}(\alpha) \in \mathbb{N}\cup \{ \infty \}$
$$ sup_{b}(\alpha)  = sup\{ m : m=0 \text{ or m is the length of a branch in } D_\alpha  \}$$

\end{dfn}

\begin{thm}
 Let $\alpha,\epsilon \in T (S)$, where S is finite, $\epsilon$ is an idempotent  and $\alpha \in C(\epsilon)$. Let ${A_1, A_2,..., A_t}$ be the set of components of $D_{\Phi_\alpha^{\epsilon}}$ and $s$ = sup$_b$$(\Phi^{\epsilon}_{\alpha})$. Let $M$ be the set of numbers of the form $k_i$ ,$1\leq i\leq t$,where $k_i$ is the length of the cycle $C_i$ in $A_i$, $p_i$ is the length of each cycle of $c$ that occurs in $C_i$, and $l_i$ is the unique number in ${0,1,..., p_i - 1}$ such that $\alpha ^{k_i }(x) = \epsilon^{l_i}(x)$, where $S$ is any element of any cycle of $c$ that occurs in $C_i$. Then $\alpha$ is a Cayley function if  the following conditions are satisfied:\begin{enumerate}
\item the largest element $m$ of $M$ is a multiple of every element of M
\item if $s>0$, then some component $A_r$ of $D_{\Phi_\alpha^{\epsilon}}$ such that $k_i = m$ has a branch
of length s.\end{enumerate}
\end{thm}
\begin{proof} Suppose that conditions (1) and (2) are satisfied. We already observed that $s = sup_b(\epsilon)$ and that $M$ is the set of the lengths of cycles in $D_\alpha$. Thus, 2(c(i)) and 2(c(ii)) of Theorem 2 hold by (1). By Lemmas 2 ,3 and 4 , $D_\alpha$ has a cycle of length m with a branch of length s. Hence 2(c(iii)) of Theorem 2 holds, and so $\alpha$ is a Cayley function.
\end{proof}

\begin{exam} let  $S=\{a_1, a_1, a_2, b,  b_1, b_2, b_3,  c, c_1, c_2, d, e, e_1 \}$ and \\ $\epsilon  = 
\left({
\begin{array}{ccccccccccccccccccccccc}a&a_1&a_2&b&b_1&b_2&b_3&c&c_1&c_2&d&e&e_1
   \\a&a&a&b&b&b&b&c&c&c&d&e&e
\end{array}}\right)$ 
 then the directed graph of $\epsilon $ is  the following
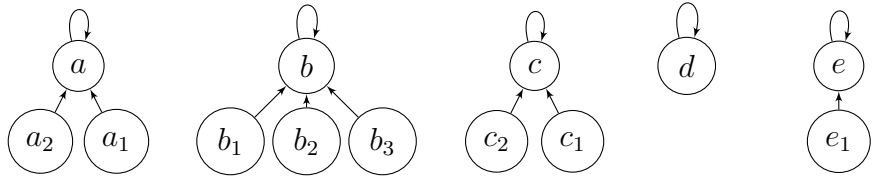
\begin {center}

\begin{tikzpicture}

\tikzset{vertex/.style = {shape=circle,draw,minimum size=0.1cm}}
\tikzset{edge/.style = {->,> = latex'}}

 \node[vertex] (n1) at (0,0) {$a $};
 \node[vertex] (n11) at (0.5,-1) {$a_1$};
 \node[vertex] (n13) at (-.5,-1) {$a_2$};

 \node[vertex] (n2) at (3,0) {$b$};
\node[vertex] (n21) at (2,-1) {$b_1$};
\node[vertex] (n22) at (3,-1) {$b_2$};
\node[vertex] (n23) at (4,-1) {$b_3$};

 \node[vertex] (n3) at (6,0) {$c$};
 \node[vertex] (n31) at (6.5,-1) {$c_1$};
  \node[vertex] (n32) at (5.5,-1) {$c_2$};

 \node[vertex] (n4) at (8,0) {$d$};

 \node[vertex] (n5) at (10,0) {$e$};
\node[vertex] (n51) at (10,-1) {$e_1$};

\draw[edge] (n1) to[loop above](n1);
\draw[edge] (n2) to[loop above](n2);
\draw[edge] (n3) to[loop above](n3);
\draw[edge] (n4) to[loop above](n4);
\draw[edge] (n5) to[loop above](n5);
\draw[edge] (n11) to (n1);
\draw[edge] (n13) to (n1);
\draw[edge] (n21) to (n2);
\draw[edge] (n22) to (n2);
\draw[edge] (n23) to (n2);
\draw[edge] (n31) to (n3);
\draw[edge] (n32) to (n3);
\draw[edge] (n51) to (n5);

\end{tikzpicture}
\end{center}

Let $\alpha =
\left({
\begin{array}{ccccccccccccccccccccccc}a&a_1&a_2&b&b_1&b_2&b_3&c&c_1&c_2&d&e&e_1
   \\b&b&b&c&c&c&c&d&d&d&e&b&b
\end{array}}\right)$ then $\alpha \in C(\epsilon)$ and let  $\Phi_\alpha^{\epsilon}$ be the as defined in equation (1) on the connected components of $\epsilon$ then $D_{\Phi_\alpha^{\epsilon}}$  is  the following
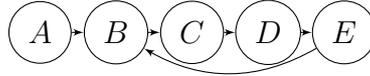
\begin {center}

\begin{tikzpicture}

\tikzset{vertex/.style = {shape=circle,draw,minimum size=.1cm}}
\tikzset{edge/.style = {->,> = latex'}}

 \node[vertex] (n1) at (0,0) {$A$};
 \node[vertex] (n2) at (1,0) {$B$};
 \node[vertex] (n3) at (2,0) {$C$};
 \node[vertex] (n4) at (3,0) {$D$};
  \node[vertex] (n5) at (4,0)  {$E$};

\draw[edge] (n1) to (n2);
\draw[edge] (n2) to (n3);
\draw[edge] (n3) to (n4);
\draw[edge] (n4) to (n5);
\draw[edge] (n5) to[bend left] (n2);

\end{tikzpicture}
\end{center}
and the directed graph of $\alpha$ is 
 
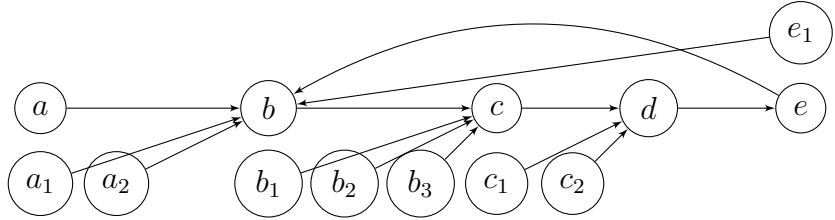
\begin {center}

\begin{tikzpicture}

\tikzset{vertex/.style = {shape=circle,draw,minimum size=.1cm}}
\tikzset{edge/.style = {->,> = latex'}}

 \node[vertex] (n1) at (0,0) {$a $};
 \node[vertex] (n11) at (0,-1) {$a_1$};
 \node[vertex] (n13) at (1,-1) {$a_2$};

 \node[vertex] (n2) at (3,0) {$b$};
\node[vertex] (n21) at (3,-1) {$b_1$};
\node[vertex] (n22) at (4,-1) {$b_2$};
\node[vertex] (n23) at (5,-1) {$b_3$};

 \node[vertex] (n3) at (6,0) {$c$};
 \node[vertex] (n31) at (6,-1) {$c_1$};
  \node[vertex] (n32) at (7,-1) {$c_2$}; 

 \node[vertex] (n4) at (8,0) {$d$};

 \node[vertex] (n5) at (10,0) {$e$};
\node[vertex] (n51) at (10,1) {$e_1$};

\draw[edge] (n1) to (n2);
\draw[edge] (n2) to (n3);
\draw[edge] (n3) to (n4);
\draw[edge] (n4) to (n5);
\draw[edge] (n5) to[bend right] (n2);

\draw[edge] (n11) to (n2);
\draw[edge] (n13) to (n2);
\draw[edge] (n21) to (n3);
\draw[edge] (n22) to (n3);
\draw[edge] (n23) to (n3);
\draw[edge] (n31) to (n4);
\draw[edge] (n32) to (n4);
\draw[edge] (n51) to (n2);

\end{tikzpicture}
\end{center}
\end{exam}

Thus we have formulated  a criterion  to decide whether a given function $\alpha\in C(\epsilon)$ is a Cayley function by analysing the components of $D_{\Phi_\alpha^{\epsilon}}$ and the numbers from $M$, provided that at least one of the cycle has no branches.

\begin{thm}
 Let $\alpha,\epsilon \in T (S)$, where S is infinite, $\epsilon$ is an idempotent  and $\alpha \in C(\epsilon)$. Let  $D_{\Phi^{\epsilon}_{\alpha}}$ have a component of type $rro$. Then $\alpha$ is Cayley if $D_{\Phi^{\epsilon}_{\alpha}}$ has a component of type $rro$ such that:
\begin{enumerate}
\item it is the join of a maximal right ray $R= [x_0 x_1 x_2.... \rangle$ and its branches;
\item for every $i\geq1$, if $[y_0 y_1 ....y_m =x_i]$ is a branch of R, then $m\leq i$.
\end{enumerate}
\end{thm}
\begin{proof} Suppose that  $D_{\Phi^{\epsilon}_{\alpha}}$ has a component $A$ of type $rro$  and that $A$ is the join of a maximal right ray $[x_0 x_1 x_2.... \rangle$ and its branches such that for every branch $[y_0 y_1 ....y_m =x_i]$ satisfies the above two condition. Now since each $\{x_0 , x_1,  x_2.... \}$  is a connected component that contains a one cycle,  the one-cycles in $\{x_0 , x_1,  x_2.... \}$  will be a maximal right ray say$[x'_0 x'_1 x'_2.... \rangle$ in $D_\alpha$ and if the branches of $\{x_0 , x_1,  x_2.... \}$ induce a branch  $[z_0 z_1 ....z_m =x'_i]$ in $D_\alpha$ then  we can see that $m\leq i$, for otherwise there should be a branch such that $m\geq i$ in $D_{\Phi^{\epsilon}_{\alpha}}$. Thus  $D_{\alpha/_{Z\cap L}}$  will have a rro that satisfies the two conditions. which implies that $\alpha  $ is Cayley 
\end{proof}

\begin{figure}[htbp]\label{rro2}
\begin{center}
\begin{tikzpicture}

\tikzset{vertex/.style = {shape=circle,fill=black,minimum size=4pt,
                            inner sep=0pt}}
\tikzset{edge/.style = {->,> = latex'}}

 \node[vertex] (n1) at (0,0)  [label=left:$x_0$] {};
 \node[vertex] (n2) at (0,1) [label=left:$x_1$] {};
  \node[vertex] (n3) at (0,2)  [label=left:$x_2$] {};
   \node[vertex] (n4) at (0,3)  [label=left:$x_3$] {};
    \node[vertex] (n5) at (0,4)  [label=left:$$] {};

     \node[vertex] (n11) at (2,0)  [label=left:$y_0$] {};
      \node[vertex] (n12) at (2,1)  [label=left:$y_1$] {};
       \node[vertex] (n13) at (4,0)  [label=left:$y_2$] {};
        \node[vertex] (n14) at (6,0)  [label=left:$y_3$] {};
         \node[vertex] (n15) at (4.5,.75)  [label=left:$y_4$] {};
          \node[vertex] (n16) at (3,1.5)  [label=left:$y_5$] {};
           \node[vertex] (n17) at (1.5,2.25)  [label=left:$y_6$] {};

\draw[edge] (n1) to (n2);
\draw[edge] (n2) to (n3);
\draw[edge] (n3) to (n4);
\draw[edge] (n4) to (n5);
\draw[edge] (n11) to (n2);
\draw[edge] (n13) to (n12);
\draw[edge] (n13) to (n3);
\draw[edge] (n14) to (n15);
\draw[edge] (n15) to (n16);
\draw[edge] (n16) to (n17);
\draw[edge] (n17) to (n4);

\end{tikzpicture}
\caption{}
\label{default}
\end{center}
\end{figure}

The converse is not true. For  if we have a rro in $D_{\Phi^{\epsilon}_{\alpha}}$ such as  the directed graph as in Fig 1 , then it does not satisfy the two condition  of the theorem but  $\alpha $ could have a $rro$ that satisfies the conditions in the theorem. For suppose that each $x_i$ and $y_j$ has exactly a one cycle and a branch and assume that each branch of $y_j$ is mapped to the cycle and each branch in $x_i$ is mapped to a branch then the branches of $x_i$ forms a right ray that satisfies the conditions of the theorem in $D_\alpha$.

\begin{lemma}
 Let $\alpha,\epsilon \in T (S)$, where S is infinite, $\epsilon$ is an idempotent  and $\alpha \in C(\epsilon)$ . Let $A$ be a component of $D_{\Phi^{\epsilon}_{\alpha}}$ that has a double ray $W = \langle... x_{-1} x_0 x_1 ...\rangle$  . 
 Then 
  \begin{enumerate} 
 \item the 1-cycles in each $x_i$ forms a double ray in $D_\alpha$.
 \item if $W$ has a finite branch of length $k$ then $D_\alpha$ will also have a finite branch of length $k$.
 \item if $W$ has an  infinite branch then $D_\alpha$ will also have an infinite branch.
   \end{enumerate}
\end{lemma} 
\begin{proof} 
 \begin{enumerate} 
\item
Suppose that  $D_{\Phi^{\epsilon}_{\alpha}}$ has a component $A$ that has a double ray $W = \langle... x_{-1} x_0 x_1 ...\rangle$  and that $A$ is the join of the double ray $W = \langle... x_{-1} x_0 x_1 ...\rangle$ and its branches. Now since each $\{.., x_{-1}, x_0 , x_1,  x_2.... \}$  is a connected component that contains a one cycle,  the one-cycles in $\{.., x_{-1}, x_0 , x_1,  x_2.... \}$ will be a double ray  say $[x'_{-1} x'_0 x'_1 x'_2.... \rangle$ in $D_\alpha$ and  the branches of $\{x_{-1} , x_0 , x_1,  x_2.... \}$ could introduce a branch of finite or infinite  length in $D_\alpha$ 
\item Suppose that $W$ has a finite branch  of length $k$ in $D_{\Phi^{\epsilon}_{\alpha}}$, then  one cycles in each component of the finite branch will form a finite branch of$[x'_{-1} x'_0 x'_1 x'_2.... \rangle$ in $D_\alpha$.
\item Suppose that $W$ has a infinite branch  of length $k$ in $D_{\Phi^{\epsilon}_{\alpha}}$. then the one cycles in each component of the infinite branch will form a finite branch of$[x'_{-1} x'_0 x'_1 x'_2.... \rangle$ in $D_\alpha$.
 \end{enumerate}

\end{proof}

Let   $W = \langle... x_{-1} x_0 x_1 ...\rangle$ be a double ray $D_{\Phi^{\epsilon}_{\alpha}}$ then the above lemma  says that  in $D_\alpha$ that there exists a double ray and further if $D_{\Phi^{\epsilon}_{\alpha}}$ has a finite branch of length $k $ then a branch of length $k$ exists but there could exist a brach of greater length or even an infinite branch.

\begin{thm} Let $\alpha,\epsilon \in T (S)$, where S is infinite, $\epsilon$ is an idempotent  and $\alpha \in C(\epsilon)$ such that the stabiliser of $\alpha $ is $s$
  Then $\alpha$ is Cayley if $D_{\Phi^{\epsilon}_{\alpha}}$ has a double ray  $W = \langle... x_{-1} x_0 x_1 ...\rangle$ such that
  \begin{enumerate}
\item if $s >0$ then $W$ has a finite branch at $x_i$ of length $s$;
\item $W$ has an infinite branch at each $x_j$ with $ j > i$.
\end{enumerate}

 \end{thm}

\begin{proof} Suppose that  $D_{\Phi^{\epsilon}_{\alpha}}$ have a double ray  $W = \langle... x_{-1} x_0 x_1 ...\rangle$, that satisfies the above two condition  then the cycle $D_{\alpha/_{Z\cap L}}$  will have has a double ray  $W = \langle... x_{-1} x_0 x_1 ...\rangle$ that satisfies the two conditions which implies that $\alpha $ is Cayley. 
\end{proof}
In general  if $D_{\Phi^{\epsilon}_{\alpha}}$ has a double ray  $W = \langle... x_{-1} x_0 x_1 ...\rangle$ $D_\alpha$ could have a component of type rro but since $\alpha$ has a stabiliser $\alpha$ will have no components of type rro.
The converse is not true. For example if we have a rro in $D_{\Phi^{\epsilon}_{\alpha}}$ such as  the directed graph is the following,

\begin{figure}[htbp]
\begin{center}
\begin{tikzpicture}

\tikzset{vertex/.style = {shape=circle,fill=black,minimum size=4pt,
                            inner sep=0pt}}
\tikzset{edge/.style = {->,> = latex'}}

 \node[vertex] (n1) at (0,0)  [label=left:$x_0$] {};
 \node[vertex] (n2) at (0,1) [label=left:$x_1$] {};
  \node[vertex] (n3) at (0,2)  [label=left:$x_2$] {};
   \node[vertex] (n4) at (0,3)  [label=left:$x_3$] {};
    \node[vertex] (n5) at (0,4)  [label=left:$$] {};
     \node[vertex] (n6) at (0,5)  [label=left:$$] {};
    \node[vertex] (n7) at (0,6)  [label=left:$$] {};
       \node[vertex] (n8) at (0,-1)  [label=left:$$] {};

     \node[vertex] (n11) at (1,1)  [label=above:$z_0$] {};
      \node[vertex] (n12) at (2,1)  [label=above:$z_1$] {};
       \node[vertex] (n13) at (3,1)  [label=above:$z_2$] {};
        \node[vertex] (n14) at (1,2)  [label=above:$y_0$] {};
         \node[vertex] (n15) at (2,2)  [label=above:$y_1$] {};
          \node[vertex] (n16) at (3,2)  [label=above:$y_2$] {};
           \node[vertex] (n17) at (4,1)  {};
            \node[vertex] (n18) at (4,2)  {};
                     \node[vertex] (n19) at (5,2)  {};
                              \node[vertex] (n20) at (6,2)  {};
                                       \node[vertex] (n21) at (5,1)  {};
                                                \node[vertex] (n22) at (6,1)  {};

\draw[edge] (n1) to (n2);
\draw[edge] (n2) to (n3);
\draw[edge] (n3) to (n4);
\draw[edge] (n4) to (n5);
\draw[edge] (n11) to (n2);
\draw[edge] (n12) to (n11);
\draw[edge] (n13) to (n12);

\draw[edge] (n14) to (n3);
\draw[edge] (n15) to (n14);
\draw[edge] (n16) to (n15);

\end{tikzpicture}
\caption{}
\label{default}
\end{center}
\end{figure}

 then it does not satisfy the two condition  of the theorem but  $\alpha $ could be Cayley. For, suppose that each $x_i$ and $y_j$ has exactly a one cycle and a branch and assume that each branch of $y_j$ is mapped to the cycle and each branch in $x_i$ is mapped to a branch then the branches of $x_i$ forms a double ray that has no infinite branch. Then $\alpha$ is Cayley if the stabiliser of $\alpha$ is 0. If $\alpha$ has stabiliser $s$ then we can assume that $x_1$ ... $x_s$ has a one cycle and two branches and each branch is mapped to distinct branches then again $\alpha$ is Cayley but $D_{\Phi^{\epsilon}_{\alpha}}$ does not satisfy the conditions of the theorem.

{  }

\end{document}